\documentclass{article}
\usepackage[utf8]{inputenc}
\usepackage{amsthm}
\usepackage{mathtools}
\usepackage{enumerate}
\usepackage{url}
\newtheorem{theorem}{Theorem}[section]
\newtheorem{corollary}{Corollary}[theorem]
\newtheorem{lemma}[theorem]{Lemma}
\newtheorem*{remark}{Remark}

\DeclarePairedDelimiter{\ceil}{\lceil}{\rceil}
\DeclarePairedDelimiter{\floor}{\lfloor}{\rfloor}

\title{A Note on Bounded Biclique Coverings of Complete Graphs}
\author{Brian Gu}
\date{August 2015}

\begin{document}

\maketitle

\begin{abstract}
    An undirected biclique $K_{a,b}$ is a graph with vertices partitioned into two sets: a set $A$ containing $a$ vertices and a set $B$ containing $b$ vertices such that every vertex in set $A$ is connected to every vertex in set $B$, and such that no two vertices in the same set have an edge between them. A well-known result is that a minimum of $\lceil \log_2{n} \rceil$ bicliques graphs of any size are needed to edge-cover the complete graph on $n$ vertices. We prove a lower bound on minimum vertex-weighted biclique coverings of the complete graph $n$, and use this to prove an asymptotic formula for the minimum number of bicliques $K_{x,x}$ with bounded component size needed to cover the complete graph on $n$ vertices.
\end{abstract}

\section{Introduction}

An \textit{undirected biclique} is a bipartite graph where every vertex in the first set is connected to every vertex in the second set. We use the notation $K_{a,b}$ to denote a biclique where the first set has $a$ vertices and the second set has $b$ vertices. The \textit{size} of the biclique $K_{a,b}$ is defined as $a+b$. Call the two \textit{components} of the biclique $K_{a,b}$ the independent sets of the biclique between which every edge exists.

The main result of this note is the following: 

\begin{theorem}
\label{mainresult}
If $m$ is the minimum number of bicliques with component size at most $x$ needed to cover a complete graph $K_n$, then $m = \Theta \left( \left(\frac{n}{x}\right)^2 + \left(\frac{n}{x}\right)\log{x} \right)$.
\end{theorem}

By convention, all logarithms are taken base $2$.

\section{Proof}

We prove the asymptotic expression in several pieces. First, we show that both $\left(\frac{n}{x}\right)^2$ and $\left(\frac{n}{x}\right)\log{x}$ are asymptotic lower bounds on the number of bicliques needed for a complete edge-covering. 

The first lower bound is easily shown by counting edges.

\begin{lemma}
\label{n2x2lower}
If $m$ is the number of bicliques with component sizes $x$ needed to edge-cover a complete graph on $n$ vertices, then
\[m \ge \frac{n(n-1)}{2x^2}.\]
\end{lemma}

\begin{proof}
Note that the total number of edges in the complete graph on $n$ vertices is $\frac{n(n-1)}{2}$. Each biclique $K_{x,x}$ has a total of $x^2$ edges. Thus, there must be at least $\frac{n(n-1)}{2x^2}$ bicliques in an edge-covering of $K_n$ using only $K_{x,x}$. 
\end{proof}

To compute the second lower bound, we use the following lemma about the sizes of bicliques.

\begin{lemma}
\label{bccover}
Consider an edge-covering of a complete graph $G$ on $x$ vertices using only bicliques. The sum of the sizes of the bicliques is at least $x\log{x}$.
\end{lemma}

\begin{proof}
Consider a partial covering of the graph $G$ with $c$ bicliques. We show that if the $c$ bicliques fully edge-cover $G$, then the sum of their sizes is at least $x\log{x}$, where $x$ is the number of vertices in $G$. 

Create a matrix $M$ with $x$ rows and $c$ columns, where each row represents a vertex and each column represents a biclique. Enumerate the bicliques $B_1, B_2, \dots, B_c$. For the $i$th biclique, consider the two components; for each vertex $v$ in the first component, assign a $0$ to the entry in the $v$th row and the $i$th column, and for each vertex $u$ in the second component, assign a $1$ to the entry in the $u$th row and the $i$th column. Assign a $*$ to the remaining entries in the column. 

Two vertices have an edge between themselves if there exists a column such that one vertex has a $1$ in the column while the other has a $0$. Call two vertices \textit{distinguishable} if such a column exists, and \textit{indistinguishable} otherwise. This is equivalent to calling two $\{0,1,*\}$ strings \textit{distinguishable} if there exists some index where one string has a $0$ and the other has a $1$, and \textit{indistinguishable} otherwise.

Now we use a pigeonhole argument to show that if every vertex is distinguishable, then $\sum{|B_i|} \ge x\log{x}$.

Consider $2^c$ holes, each representing a unique bitstrings of length $c$. For each vertex $v$, place a copy of vertex $v$ in every hole where its row string--formed by considering the $v$th row of $M$ as a $\{0,1,*\}$ string--is indistinguishable from the hole's bitstring. 

If two vertices are placed in the same hole, then they are indistinguishable and the partial covering is not a complete edge-covering because indistinguishable vertices do not have an edge between them; otherwise, one would have a $0$ at an index where another has a $1$, and they would not be indistinguishable with the same $\{0,1\}$ bitstring. Thus, by the pigeonhole principle, the total number of vertex copies placed in holes must be at most the number of holes itself, $2^c$. 

We now count the number of vertex copies. Suppose that vertex $v$'s row has $s_v$ stars, and that the total number of stars in the matrix is $s$. Then $v$'s bitstring is indistinguishable from $2^{s_v}$ bitstrings of length $c$, so a total of $2^{s_v}$ copies of $v$ are placed in holes. Thus the total number of vertex copies over all vertices is $\sum\limits_{i=1}^x 2^{s_i}$. If the partial covering completely covers $G$, this sum is at most $2^c$. However, by the convexity of the function $2^n$, we know that 
\[x2^{\frac{s}{x}} \le \sum\limits_{i=1}^x 2^{s_i} \le 2^c.\]
Therefore
\[x2^{\frac{s}{x}} \le 2^c.\]
\[2^{\log{x} + \frac{s}{x}} \le 2^c.\]
Taking the logarithm,
\[\log{x} + \frac{s}{x} \le c,\]
which can be rewritten as
\[cx - s \ge x\log{x}.\]

But $cx - s$ is precisely the number of $0$s and $1$s in $M$, and is thus the sum of the sizes of all the bicliques in the complete edge-covering. Thus 
\[\sum{|B_i|} \ge x\log{x},\]
as desired.
\end{proof}

We also use the result on the minimum number of bicliques needed to cover a complete graph. This result is well-known; we were made aware of it through a USA TST problem \cite{AoPS} and a conversation with Evan Chen.

\begin{corollary}
\label{ubcover}
The minimum number of bipartite graphs of any size needed to cover a complete graph $G$ on $x$ vertices is $\ceil*{\log{x}}$.
\end{corollary}

\begin{proof}
By Lemma \ref{bccover}, the sum of the sizes of the bicliques in an edge-covering of $G$ using only bicliques is at least $x\log{x}$. Because every biclique in a graph with only $x$ vertices can have size at most $x$, the number of bicliques needed is at least $\frac{x\log{x}}{x}=\log{x}$. Since an integer number of bicliques must be used, the number of bicliques needed is at least $\ceil*{\log{x}}$.

Now we show that we can use exactly $\ceil*{\log{x}}$ bicliques to edge-cover $G$. Assign to each vertex of $G$ a unique bitstring of length $\ceil*{\log{x}}$. Now consider $\ceil*{\log{x}}$ bicliques, each generated in the following way:
\begin{itemize}
\item For the $i$th biclique, connect every vertex with a $0$ in the $i$th index of the vertex's bitstring to every vertex with a $1$ in the $i$th index.
\end{itemize}
Since every bitstring is distinct, for every pair of vertices there exists some index where their digits differ. Thus, every pair of vertices has an edge between them and the $\ceil*{\log{x}}$ bicliques completely edge-cover the graph $G$.
\end{proof}

The following related corollary will be used later to prove the upper bound.

\begin{corollary}
\label{equalcover}
There exists a covering of $K_{2x}$ using $\ceil*{\log{2x}}$ subgraphs $K_{x,x}$.
\end{corollary}

\begin{proof}
First, consider any minimal covering of a $K_{x}$ using the procedure in Corollary \ref{ubcover}, and the associated set of bitstrings $B$. Now consider $C$, the set of complements of these bitstrings (bitstrings are complements if they have the same length and do not have the same character at any index). Construct a set $B'$ of the same size as $B$ by appending a $0$ to the front of each bitstring in $B$. Construct a set $C'$ by appending a $1$ to the front of each bitstring in $C$. The intersection of $B'$ and $C'$ is clearly empty, so their union has $2x$ bitstrings. Now consider $A$, the union of $B'$ and $C'$. Because $B$ and $C$ are complements, for every index except the first, the number of strings with a $0$ at that index in $A$ is equal to the number of strings with a $1$ at that index. Finally, by the construction of $B'$ and $C'$, there are an equal number of $0$s and $1$s in the first index. It follows from here that $A$ encodes a covering of a $K_{2x}$ using $\ceil*{\log{x}} + 1 = \ceil*{\log{2x}}$ subgraphs $K_{x,x}$, and that $A$ is a minimal covering of a $K_{2x}$.
\end{proof}

With Lemma \ref{bccover} and Corollary \ref{ubcover}, we now present a proof that the second term in the asymptotic formula is a lower bound.

\begin{lemma}
\label{loglower}
If $m$ is the number of bicliques with component size $x$ needed to completely edge-cover a complete graph on $n$ vertices, then
\[m \ge \frac{n\log{2x}}{2x}.\]
\end{lemma}

\begin{proof}
From Lemma \ref{bccover}, the sum of sizes of bicliques in a covering of $K_n$ with only bicliques must be at least $n\log{n}$. If all of the bicliques have component sizes $x$, then each biclique has total size $2x$, so there must be at least $\frac{n\log{n}}{2x}$ bicliques, which is at least $\frac{n\log{2x}}{2x}$.
\end{proof}

\begin{remark}
An alternate proof of Lemma \ref{loglower} motivates the construction of the upper bound but gives a weaker lower bound. 

Let $G$ be a $K_n$, edge-covered with bicliques $Y_1, Y_2, \dots, Y_m$, where each $Y_i$ is a $K_{x,x}$. Dividing $G$ into roughly $\frac{n}{x}$ disjoint sets $X_1, X_2, \dots, X_{\floor{\frac{n}{x}}}$ of roughly $x$ vertices each, we can apply Lemma \ref{bccover} and Corollary \ref{ubcover} to find that the sum of the sizes of the intersections of all $Y_i$ with a fixed $X_j$ is at least $x\log{x}$. Since the $X_j$ are disjoint, the sum of the sizes of all $Y_i$ can be shown to be at least $\floor*{\frac{n}{x}}x\log{x}$. Thus the number of bicliques in the edge-covering is at least $\frac{1}{2}\floor*{\frac{n}{x}}\log{x}$.
\end{remark}

Finally, we compute an upper bound on $m$ via construction, which is asymptotically equivalent to the lower bound. 

\begin{lemma}
\label{upperconstruct}
There exists a complete edge-covering of a complete graph on $n$ vertices using at most
\[4\binom {\ceil*{\frac{n}{2x}}} {2} + \ceil*{\frac{n}{2x}}\ceil*{\log{2x}}\]
complete bipartite graphs $K_{x,x}$ of component size $x$.
\end{lemma}

\begin{proof}
To impose this upper bound on $m$, we present the following construction. Let $G$ be the complete graph on $n$ vertices we wish to cover.

Motivated by the second proof of Lemma \ref{loglower}, we partition the vertices of $G$ into $\ceil*{\frac{n}{2x}}$ disjoint groups of at most $2x$ vertices each. It takes four $K_{x,x}$ to connect every edge going between any two of the groups; there are $\ceil*{\frac{n}{2x}}$ groups, so $4\binom {\ceil*{\frac{n}{2x}}} {2}$ bipartite graphs can be used to cover every edge that goes between all pairs of vertices not in the same group. The remaining uncovered edges now only exist within groups of at most $2x$ vertices each. But by corollary \ref{equalcover}, we know that each of these can be edge-covered with $\ceil{\log{2x}}$ bipartite graphs each; since there are $\ceil*{\frac{n}{2x}}$ of these, our construction requires 
\[4\binom {\ceil*{\frac{n}{2x}}} {2} + \ceil*{\frac{n}{2x}}\ceil*{\log{2x}}\]
bipartite graphs to completely edge-cover the $K_n$ graph $G$. 
\end{proof}

With the upper and lower bounds shown, we prove the main result of this section.

\begin{proof}[Proof of Theorem \ref{mainresult}]
From Lemma \ref{n2x2lower} and the stronger proof of Lemma \ref{loglower}, we obtain 
\[\max{\left\{\frac{n(n-1)}{2x^2},\frac{n\log{2x}}{2x}\right\}} \le m. \]
From Lemma \ref{upperconstruct}, we obtain
\[m \le 4\binom {\ceil*{\frac{n}{2x}}} {2} + \ceil*{\frac{n}{2x}}\ceil*{\log{2x}}.\]
Both the lower and the upper bounds are $\Theta \left( \left(\frac{n}{x}\right)^2 + \left(\frac{n}{x}\right)\log{x} \right)$, as desired, proving the theorem.
\end{proof}

\section{Concluding Remarks}
\begin{itemize}
\item We can also easily show that for all $\epsilon > 0$, the upper and lower bounds are provably within a factor of $6+\epsilon$ of each other when $n$ is sufficiently large. Additionally, for all $\epsilon > 0$, the upper and lower bounds are provably within a factor of $2+\epsilon$ of each other when both $n$ and $\frac{n}{x}$ are sufficiently large. 
\item The second term of the asymptotic formula dominates when $x \gg \frac{n}{\log{n}}$, and the first term dominates when $x \ll \frac{n}{\log{n}}$. This is consistent for the expressions for the minimum number of $K_{1,1}$ and $K_{n/2, n/2}$ graphs needed to cover a $K_n$, respectively.
\end{itemize}

We would like to thank Brandon Tran for his guidance and mentorship, Evan Chen for introducing us to the result on unbounded biclique coverings of the complete graph, Po-Shen Loh for his advice, and the Center for Excellence in Education and the Research Science Institute at MIT for providing us with an opportunity for research.

\bibliography{bibliography}{}
\bibliographystyle{plain}

\end{document}